\documentclass[11pt]{article}
\usepackage{amsmath,amssymb,amsthm}
\usepackage{hyperref} 
\theoremstyle{definition} 
\newtheorem{definition}{Definition}[section]
\theoremstyle{plain}
\newtheorem{theorem}[definition]{Theorem}
\newtheorem{proposition}[definition]{Proposition}
\newtheorem{remark}[definition]{Remark}
\newtheorem{lemma}[definition]{Lemma}
 \title{Duality between Bott--Chern and Aeppli Cohomology on Non-Compact Complex Manifolds} 
\author{Xiaojun Wu} 
\date{\today} 
\begin{document} 
\def\A{\mathcal{A}}
\def\cI{\mathcal{I}}
\def\Z{\mathbb{Z}}
\def\Q{\mathbb{Q}}  \def\C{\mathbb{C}}
 \def\R{\mathbb{R}}
 \def\N{\mathbb{N}}
 \def\H{\mathbb{H}}
  \def\P{\mathbb{P}}
 \def\rC{\mathrm{C}}
  \def\d{\partial}
 \def\dbar{{\overline{\partial}}}
\def\dzbar{{\overline{dz}}}
\def \ddbar {\partial \overline{\partial}}
\def\cB{\mathcal{B}}
\def\cD{\mathcal{D}}  \def\cO{\mathcal{O}}
\def\cbarO{\overline{\mathcal{O}}}
\def\D{\mathcal{D}}
\def\cC{\mathcal{C}}
\def\cF{\mathcal{F}}
\def \rank{\mathrm{rank}}
\def \deg{\mathrm{deg}}
\def \tot{\mathrm{Tot}}
\def \id{\mathrm{id}}
\bibliographystyle{plain}
\def \End{\mathrm{End}}
\def \dim{\mathrm{dim}}
\def \div{\mathrm{div}}
\def \ker{\mathrm{Ker}}
\def \im{\mathrm{Im}}
\def \rC{\mathrm{Cone}}
\newcommand{\Ub}{\mathcal{U}}
\newcommand{\dcech}{\check{\delta}}
\newcommand{\dcechg}{\delta\!\!\!\check{\delta}}
\newcommand{\lc}{\mathcal{L}}
\newcommand{\ec}{\mathcal{E}}
\maketitle 
\begin{abstract}In this paper we establish duality theorems relating Bott–Chern and Aeppli cohomology, both with and without compact support, on non-compact complex manifolds under suitable pseudoconvexity assumptions. In particular, on Stein manifolds we obtain a full Bott–Chern–Aeppli duality extending Serre duality for Dolbeault cohomology. We also show that these results fail in general without pseudoconvexity assumptions by constructing explicit counterexamples on non-compact complex surfaces. \end{abstract}
\section{Introduction}
Cohomological duality plays a central role in geometry, as it provides an effective framework for computing and relating cohomology groups associated with a given manifold. Conversely, the failure of a duality often reflects intrinsic geometric or analytic features of the underlying space. A classical example is provided by Hodge duality: its failure implies that a compact complex manifold does not admit a K\"ahler structure. 

On smooth manifolds, a natural duality is exemplified by Poincar\'e duality, which relates cohomology groups in complementary degrees. When the underlying manifold is non-compact, however, such duality statements become more subtle and generally fail without additional assumptions. This naturally leads to the introduction of cohomology theories with compact support, which are designed to recover meaningful duality statements in the non-compact setting. 

On a complex manifold, de~Rham cohomology admits a refinement through Dolbeault cohomology. The Dolbeault cohomology groups $H^{p,q}_{\bar\partial}(X)$ encode the complex structure of the manifold and, when the $n$-dimensional complex manifold $X$ is compact, fit into Serre duality, yielding a perfect pairing between $H^{p,q}_{\bar\partial}(X)$ and $H^{n-p,n-q}_{\bar\partial}(X)$. In the non-compact case, the situation is more delicate. As shown by Serre~\cite[Section~14]{Ser55}, Serre duality may fail on general non-compact complex manifolds, even when cohomology with compact support is taken into account. In fact, Serre proved that duality holds only under additional functional-analytic assumptions.

Bott--Chern and Aeppli cohomologies provide refinements of Dolbeault cohomology that capture finer analytic information on complex manifolds. For completeness, we briefly recall their definitions (see, for example, \cite{Aep62,BC65,agbook,Sch07}). 
\begin{definition}\label{BC} Let $X$ be a complex manifold. The \emph{Bott--Chern cohomology} of $X$ is defined by \[ H^{p,q}_{BC}(X) := \frac{ \ker \partial \cap \ker \bar\partial \subset \mathcal A^{p,q}(X)} { \operatorname{Im}\,\partial\bar\partial \subset \mathcal A^{p,q}(X)}. \] The \emph{Bott--Chern cohomology with compact support} is defined by \[ H^{p,q}_{BC,c}(X) := \frac{ \ker \partial \cap \ker \bar\partial \subset \mathcal A^{p,q}_c(X)} { \operatorname{Im}\,\partial\bar\partial \subset \mathcal A^{p,q}_c(X)}, \] where $\mathcal A^{p,q}_c(X)$ denotes the space of smooth $(p,q)$-forms on $X$ with compact support. \end{definition}
 \begin{definition} \label{Ae} Let $X$ be a complex manifold. The \emph{Aeppli cohomology} of $X$ is defined by \[ H^{p,q}_{A}(X) := \frac{ \ker \partial\bar\partial \subset \mathcal A^{p,q}(X)} { \operatorname{Im}\,\partial + \operatorname{Im}\,\bar\partial }. \] The \emph{Aeppli cohomology with compact support} is defined by \[ H^{p,q}_{A,c}(X) := \frac{ \ker \partial\bar\partial \subset \mathcal A^{p,q}_c(X)} { \operatorname{Im}\,\partial + \operatorname{Im}\,\bar\partial }. \] \end{definition} 
 
On compact complex manifolds, Bott--Chern and Aeppli cohomology groups are finite dimensional and are in perfect duality (see, for example, \cite[Page~16]{Sch07}), thereby generalizing Serre duality in the Dolbeault setting. In contrast, the non-compact case is far less understood. While Dolbeault cohomology with compact support provides a natural framework for extending Serre duality, it is not a priori clear whether an analogous duality holds for Bott--Chern and Aeppli cohomology, either with or without compact support. Since Bott--Chern and Aeppli cohomologies encode finer analytic information than Dolbeault cohomology, understanding their duality properties on non-compact complex manifolds is a natural and subtle problem. 

In this paper, we address this question under suitable pseudoconvexity and finiteness assumptions. Using the Andreotti--Grauert finiteness theorem~\cite{AG62} together with Serre's functional-analytic criterion for duality~\cite{Ser55}, we establish duality results (see Theorem \ref{thm1}, \ref{thm2}) between Bott--Chern and Aeppli cohomology groups on strongly $q$-convex complex manifolds (see Definition \ref{qconvex}). We also observe that the expected duality holds for theta toroidal groups. By revisiting the cohomological computations in Serre~\cite[Section~14]{Ser55}, we show that, in general, the expected Bott--Chern--Aeppli duality fails on non-compact complex manifolds when the above pseudoconvexity and finiteness assumptions are not satisfied. Such an example already occurs in complex dimension~$2$. By contrast, such a failure cannot occur in complex dimension~$1$. Indeed, by the Behnke--Stein theorem, any non-compact Riemann surface is Stein, and our main theorem shows that the relevant duality holds on Stein manifolds. 

The paper is organized as follows. In Section~2 we recall the basic notions and classical results that will be used later. Section~3 is devoted to the proof of the duality results under pseudoconvexity assumptions. In Section~4 we construct explicit counterexamples illustrating the failure of Bott--Chern--Aeppli duality in the absence of such assumptions.
\paragraph{}
\textbf{Acknowledgements.} I would like to thank my postdoctoral advisor, Professor Takayuki Koike, and Dr.~Jinichiro Tanaka for several helpful and stimulating discussions related to this work. This research was supported by the JSPS Postdoctoral Fellowships for Research in Japan (Standard). I am also grateful to Osaka Metropolitan University and the Osaka Central Advanced Mathematical Institute (OCAMI) for providing an excellent research environment.
\section{Andreotti--Grauert finiteness, Serre’s criterion, and sheaf-theoretic descriptions} 
In this section, we collect some well-known results that will be used in later sections.

The following finiteness theorem is due to Andreotti and Grauert~\cite{AG62}. Throughout, we use the notation of Demailly~\cite[Chap.~IX]{agbook}.
\begin{definition}
\label{qconvex}
 A smooth function on a complex manifold $X$ of dimension $n$ is said to be \emph{strongly $q$-convex} at a point $x \in X$ if $\sqrt{-1}\partial\bar\partial\psi(x)$ has at least $n-q+1$ strictly positive eigenvalues, or equivalently, if there exists an $(n-q+1)$-dimensional subspace $F \subset T_xX$ on which the complex Hessian is positive definite. A complex manifold $X$ is said to be \emph{strongly $q$-convex} if it admits a smooth exhaustion function $\psi$ which is strongly $q$-convex outside an exceptional compact set $K \subset X$. We say that $X$ is \emph{strongly $q$-complete} if $\psi$ can be chosen such that $K=\varnothing$. \end{definition} 
\begin{theorem}[{\cite[Chap.~IX, Th.~4.10]{agbook}}]
\label{AG} 
Let $\mathcal F$ be a coherent analytic sheaf over a strongly $q$-convex complex manifold $X$. Then $H^k(X,\mathcal F)$ is Hausdorff and finite dimensional for all $k \ge q$. \end{theorem}
The criterion allowing one to pass from duality at the level of global forms to duality at the level of cohomology is given by the following classical result of Serre.
Recall that a \emph{homomorphism} between locally convex topological vector spaces is a continuous linear map $u : X \to Y$ such that the induced map $u : X \to \operatorname{Im} u$, where $\operatorname{Im} u$ is endowed with the subspace topology induced by $Y$, is an open mapping.
 \begin{proposition}[{\cite[Lemma~1]{Ser55}}]
\label{Ser1}  
Let $L,M,N$ be Fr\'echet spaces, and let \[ L \xrightarrow{u} M \xrightarrow{v} N \] be two linear continuous homomorphisms such that $v \circ u = 0$. Let $L^*,M^*,N^*$ be the topological duals, and let ${}^tu : M^* \to L^*$ and ${}^tv : N^* \to M^*$ denote the transpose maps. Set \[ C = v^{-1}(0), \quad B = u(L), \quad H = C/B, \] \[ C' = ({}^tv)^{-1}(0), \quad B' = {}^tu(N^*), \quad H' = C'/B'. \] Then $H$ is a Fr\'echet space whose topological dual is isomorphic to $H'$. \end{proposition} \begin{proposition}[{\cite[Lemma~2]{Ser55}}]
\label{Ser2} 
Let $u$ be a continuous linear map from a Fr\'echet space $L$ into a Fr\'echet space $M$. If $u(L)$ is a subspace of $M$ of finite codimension, then $u$ is a homomorphism. \end{proposition} 
We now recall the sheaf-theoretic interpretation of Bott--Chern cohomology due to Demailly (see~\cite[Chap.~VI, \S12.1]{agbook},\cite{Sch07}).
This sheaf-theoretic description provides a natural framework for relating Bott--Chern and Aeppli cohomology to Dolbeault cohomology and de~Rham cohomology.
 \begin{proposition}
 \label{DemSch}
The Bott--Chern cohomology group defined in Definition \ref{BC} is the hypercohomology group \[ H^{p,q}_{BC}(X) = \mathbb H^{p+q}(X,\mathcal B^\bullet_{p,q}) \] of the Bott--Chern complex \[ \mathcal B^\bullet_{p,q} : (2\pi \sqrt{-1})^p\mathbb C \xrightarrow{\Delta} \mathcal O \oplus \mathcal O \to \Omega^1 \oplus \Omega^1 \to \cdots \to \Omega^{p-1} \oplus \Omega^{p-1} \to \Omega^p \to \cdots \to \Omega^{q-1} \to 0, \] where $\Delta$ is multiplication by $1$ on the first component and by $-1$ on the second component. 

Moreover, the Aeppli cohomology group defined in Definition \ref{Ae} is the hypercohomology group
\[ H^{p,q}_A(X) = \mathbb H^{p+q+1}(X,\mathcal B^\bullet_{p+1,q+1}). \]  
 \end{proposition} 
The second statement was observed by Schweitzer~\cite[Section 4.d]{Sch07}.  
 
The Bott--Chern complex is quasi-isomorphic to the complex $\mathcal L^\bullet_{p,q}$, defined as in \cite[Chap.~VI, \S12.1]{agbook} as follows. 

For $k \le p+q-2$, set \[ \mathcal L^k_{p,q} = \bigoplus_{\substack{r+s=k \\ r<p,\ s<q}} \mathcal E^{r,s}, \] while for $k \ge p+q$, set \[ \mathcal L^{k-1}_{p-1,q-1} = \bigoplus_{\substack{r+s=k \\ r\geq p,\ s\geq q}} \mathcal E^{r,s}. \] The differential of $\mathcal L^\bullet_{p,q}$ is given by \[ \mathcal L^0 \xrightarrow{\;\operatorname{pr}_{\mathcal L^1}\circ d\;} \mathcal L^1 \xrightarrow{\;\operatorname{pr}_{\mathcal L^2}\circ d\;} \cdots \to \mathcal L^{k-2} \xrightarrow{\;\frac{\sqrt{-1}}{2\pi}\partial\bar\partial\;} \mathcal L^{k-1} \xrightarrow{\;d\;} \mathcal L^{k} \xrightarrow{\;d\;} \cdots . \]
Here $\mathcal E^{r,s}$ denotes the sheaf of smooth $(r,s)$-forms on $X$.
\begin{remark}
\label{topo}
Note that one may replace the sheaf of smooth differential forms in the complex \(\mathcal L^\bullet\) by the sheaf of currents. 
Since both the sheaves of smooth forms and of currents are soft, the corresponding complexes of global sections compute the Bott–Chern and Aeppli cohomology groups (with or without compact support).
The canonical inclusion of smooth forms into currents defines a quasi-isomorphism between the corresponding complexes of sheaves, and therefore induces an isomorphism on hypercohomology at the level of complex vector spaces.

In what follows, we endow Bott–Chern and Aeppli cohomology with the topology induced by the natural Fréchet topology on the space of smooth forms. 
All differentials in the complex are induced by \(d\), \(\partial\), and \(\bar\partial\), which define continuous linear operators between Fréchet spaces of smooth forms. 
However, continuity alone does not imply that these operators are homomorphisms in the sense of Serre, since their images need not be closed. 

For Bott–Chern and Aeppli cohomology with compact support, we instead endow the cohomology groups with the natural locally convex topology induced from the space of currents with compact support. 
Although the canonical inclusion of smooth forms into currents induces an isomorphism on hypercohomology as complex vector spaces, this isomorphism is in general not an isomorphism of topological vector spaces when the cohomology groups are infinite dimensional.
This discrepancy is the reason for our specific choice of topologies in the sequel.
\end{remark}

 \section{Finiteness and duality results}
In this section, we state the main results of this paper. We begin by establishing several preliminary lemmas.
  \begin{lemma} 
\label{lem1}  
Let $X$ be a strongly $r$-convex complex manifold, and let $\mathcal A^\bullet_p$ be the complex \[ \mathcal O \to \Omega^1 \to \cdots \to \Omega^{p-1} \to 0 \] starting in degree $0$. Then $H^k(X,\mathcal A^\bullet_p)$ is finite dimensional for all $k \ge p+r-1$. \end{lemma} \begin{proof} 
By the Andreotti--Grauert finiteness theorem (Theorem \ref{AG}), the cohomology groups of any coherent sheaf on $X$ are finite dimensional in degrees greater than or equal to $r$. In particular, since $\mathcal A^\bullet_1 \simeq \mathcal O_X$, the statement holds for $p=1$. 

We now proceed by induction on $p$. Assume that the claim holds for $\mathcal A^\bullet_p$. Consider the short exact sequence of complexes \[ 0 \longrightarrow \Omega^{p}[p] \longrightarrow \mathcal A^{\bullet}_{p+1} \longrightarrow \mathcal A^{\bullet}_{p} \longrightarrow 0, \] where $\Omega^{p}[p]$ denotes the sheaf $\Omega^{p}$ placed in cohomological degree $p$. The associated long exact sequence in hypercohomology yields, for each $k \ge p+r$, 
\[ {H}^{k-p}(X,\Omega^{p}) \longrightarrow \mathbb{H}^k(X,\mathcal A^{\bullet}_{p+1}) \longrightarrow \mathbb{H}^k(X,\mathcal A^{\bullet}_{p}). \] Since $\Omega^{p}$ is a coherent sheaf, the group $ H^{k-p}(X,\Omega^{p})$ is finite dimensional for $k-p \ge r$, and by the induction hypothesis $\mathbb{H}^k(X,\mathcal A^{\bullet}_{p})$ is finite dimensional for $k \ge p+r$. It follows that $\mathbb{H}^k(X,\mathcal A^{\bullet}_{p+1})$ is finite dimensional for all $k \ge p+r$, which completes the induction and hence the proof.   \end{proof} 
  \begin{lemma} 
\label{lem2}    
  Let $X$ be a strongly $r$-convex complex manifold. Assume that the Betti numbers of $X$ are finite. Then $ \mathbb{H}^k(X,\mathcal B^\bullet_{p,q})$ is finite dimensional for all $k \ge \max(p,q)+r$. \end{lemma} 
  \begin{proof}We consider, for $k \ge \max(p,q)+r$, the long exact sequence in hypercohomology \begin{equation}\label{eq:les} \mathbb{H}^{k-1}\!\left(X,\mathcal A^{\bullet}_p \oplus \overline{\mathcal A^{\bullet}_q}\right) \longrightarrow \mathbb{H}^{k}\!\left(X,\mathcal B^{\bullet}_{p,q}\right) \longrightarrow \mathbb{H}^{k}\!\left(X,\mathbb C(p) \right), \end{equation} associated with the short exact sequence of complexes \[ 0 \longrightarrow \bigl(\mathcal A^{\bullet}_p \oplus \overline{\mathcal A^{\bullet}_q}\bigr)[1] \longrightarrow \mathcal B^{\bullet}_{p,q} \longrightarrow \mathbb C(p):=(2\pi \sqrt{-1})^p\mathbb C \longrightarrow 0. \] Since $\mathcal A^{\bullet}_p$ and $\overline{\mathcal A^{\bullet}_q}$ satisfy the finiteness property established in Lemma \ref{lem1}, the group $\mathbb{H}^{k-1}\!\left(X,\mathcal A^{\bullet}_p \oplus \overline{\mathcal A^{\bullet}_q}\right)$ is finite dimensional for $k \ge \max(p,q)+r$. Moreover, $\mathbb{H}^{k}\!\left(X,\mathbb C(p)\right) \simeq H^{k}(X,\mathbb C)$ is finite dimensional by the topological assumption on $X$. It follows that $\mathbb{H}^{k}\!\left(X,\mathcal B^{\bullet}_{p,q}\right)$ is finite dimensional for all $k \ge \max(p,q)+r$, which completes the proof. 
  \end{proof} 
We emphasize that the assumption on the finiteness of Betti numbers is nontrivial. For instance, any non-compact Riemann surface is Stein, and hence strongly $1$-convex; nevertheless, its first Betti number may be infinite. A concrete example is provided by the complex plane $\mathbb C$ with an infinite discrete set of points removed. On the other hand, by~\cite[Theorem~5.2.7]{Hor90}, the Betti numbers of an $n$-dimensional Stein manifold satisfy $b_k = 0$ for all $k \ge n+1$.  
  
We are now in a position to state the main results.
  \begin{theorem}
  \label{thm1}
Let $X$ be a strongly $r$-convex complex manifold of dimension $n$ with finite Betti numbers, and assume that $p,q \ge r$. Then the Bott--Chern cohomology group $H^{p,q}_{BC}(X)$ is finite dimensional and is canonically isomorphic to the topological dual of $H^{n-p,n-q}_{A,c}(X)$. Similarly, the Aeppli cohomology group $H^{p,q}_{A}(X)$ is finite dimensional and is canonically isomorphic to the topological dual of $H^{n-p,n-q}_{BC,c}(X)$.
\end{theorem}
\begin{proof} 
We prove the Bott–Chern statement; the Aeppli case is analogous.
By Proposition \ref{DemSch} and Remark \ref{topo}, \(H^{p,q}_{BC}(X)\) is computed as the cohomology of the complex of global sections of \(\mathcal L^{\bullet}_{p,q}\), which is a complex of Fréchet spaces with continuous differentials. By Lemma \ref{lem2}, under the assumptions \(p,q\ge r\), the cohomology group \(H^{p,q}_{BC}(X)\) is finite dimensional. Hence the image of the preceding differential in \(\mathcal L^{\bullet}_{p,q}(X)\) has finite codimension in its kernel. Since kernels of continuous linear maps between Fréchet spaces are closed, Proposition \ref{Ser2} implies that this differential is a homomorphism in the sense of Serre. Therefore the hypotheses of Proposition \ref{Ser1} are satisfied for the relevant three-term subcomplex of \(\mathcal L^{\bullet}_{p,q}(X)\), and the topological dual of \(H^{p,q}_{BC}(X)\) is canonically identified with the corresponding dual cohomology group. Under the integration pairing between forms and currents, this dual group is \(H^{n-p,n-q}_{A,c}(X)\).
Note that in this case the cohomology groups are finite dimensional; consequently, the canonical inclusion of smooth forms into currents induces an isomorphism not only of complex vector spaces but also of topological vector spaces.
\end{proof} 
In the case $r=1$, a stronger duality statement can be obtained, as follows.
\begin{theorem}
\label{thm2}
Let $X$ be a strongly $1$-convex complex manifold of dimension $n$. Assume that the Betti numbers of $X$ are finite. Then, for all $p,q$, $H^{p,q}_{BC}(X)$ is the topological dual of $H^{n-p,n-q}_{A,c}(X)$, and $H^{p,q}_A(X)$ is the topological dual of $H^{n-p,n-q}_{BC,c}(X)$.
\end{theorem}
\begin{proof} 
By Theorem \ref{thm1}, the only remaining cases are $p=0$ or $q=0$. By complex conjugation symmetry, it suffices to consider $q=0$. 

In the Bott--Chern case, Lemma~\ref{lem2} implies that $\mathbb{H}^{p+1}(X,\mathcal{B}^\bullet_{p,0})$ is finite dimensional. In order to apply Proposition~\ref{Ser1}, it remains to verify that the relevant morphism in the complex $\mathcal L^\bullet$ is a homomorphism. In this case, the morphism reduces to \[ 0 \longrightarrow \mathcal E^{p,0}(X), \] which is the zero map and hence trivially a homomorphism
since the induced map to its image is open (as the image is \(\{0\}\)).
  
  In the Aeppli case with $p \ge 1$, Lemma~\ref{lem2} implies that $\mathbb{H}^{p+2}(X,\mathcal{B}^\bullet_{p+1,1})$ is finite dimensional. In order to apply Proposition~\ref{Ser1}, it remains to verify
 that the relevant morphism in the complex $\mathcal L^\bullet$ is a homomorphism. In this case, the morphism reduces to \[\mathcal  E^{p-1,0}(X) \xrightarrow{\partial}\mathcal  E^{p,0}(X). \]
 Consider the complex \[\mathcal  E^{p-1,0}(X) \xrightarrow{\partial}\mathcal  E^{p,0}(X) \xrightarrow{\partial}\mathcal  E^{p+1,0}(X). \] 
Its cohomology at the middle term by complex conjugation is given by \[ H^{p,0}_{\partial}(X) = \overline{H^{0,p}_{\bar\partial}(X)} = \overline{H^{p}(X,\mathcal O_X)}, \] which is finite dimensional for $p \ge 1$ by Theorem \ref{AG}. 
In particular, the finite dimensionality of this cohomology group implies that $\operatorname{Im}(\partial)$ has finite codimension in $\ker(\partial) \subset \mathcal E^{p,0}(X)$.
Since $\ker(\partial)$ is closed in $\mathcal E^{p,0}(X)$, it is a Fr\'echet space. Moreover, for $p \ge 1$, the image $\operatorname{Im}(\partial)$ has finite codimension in $\mathcal E^{p,0}(X)$ and is therefore also a Fr\'echet space with the topology induced from $\mathcal E^{p,0}(X)$. Consequently, $\partial$ defines a homomorphism in the sense of Serre, and Proposition~\ref{Ser1} applies, which completes the proof for the case $p \ge 1$.
 
In the Aeppli case with $p=0$, the morphism in the complex $\mathcal L^\bullet$ that needs to be verified in order to apply Proposition~\ref{Ser1} reduces to the zero map, which is trivially a homomorphism. \end{proof} 
\begin{remark}
By the works of \cite{Ume} and \cite{Kaz84},
all de Rham and Dolbeault cohomologies of a theta toroidal group are finite dimensional.
By the same arguments as above,
one can show that the topological duality between Bott--Chern cohomology with compact support and Aeppli cohomology
and the topological duality between Aeppli cohomology with compact support and Bott--Chern cohomology hold since all relevant Fréchet maps are homomorphisms because all cohomology groups involved are finite dimensional.
\end{remark}
As already observed in \cite[Section 14]{Ser55}, in general there is no duality between Dolbeault cohomology and Dolbeault cohomology with compact support. An analogous phenomenon is expected for Bott–Chern and Aeppli cohomologies. 
In the following, by refining Serre’s calculations, we present two counterexamples showing, respectively, that there is no duality between Bott–Chern cohomology with compact support and Aeppli cohomology, and between Bott–Chern cohomology and Aeppli cohomology with compact support. 
The construction is based on the hypercohomological descriptions of Bott--Chern and Aeppli cohomology and on a detailed analysis of the Hodge--to--de~Rham spectral sequence, refining the arguments in \cite[Section~14]{Ser55}.
For the convenience of the reader, we provide more details than in \cite[Section 14]{Ser55}, although these considerations are well known to experts.

\section{Counterexamples}

We recall the following calculation of Dolbeault cohomology, with and without compact support, due to Serre. The same arguments apply to $X$ the complement in $\C^2$ of any non-compact connected closed subset that is not a domain of holomorphy and contains no $2$-dimensional compact irreducible component. In what follows, we choose different closed subsets with simpler homotopy types in order to facilitate the analysis of the Hodge--to--de~Rham spectral sequence and to exhibit two distinct types of failure of duality.

We begin with ordinary Dolbeault cohomology. Since the sheaves of holomorphic differential forms are locally free, it suffices to determine the cohomology groups  \[  H^1(X,\mathcal O_X)\quad\text{and}\quad H^2(X,\mathcal O_X).  \]
First, by a theorem of Malgrange (see, for example, \cite[Corollary~(4.15), Chapter~IX]{agbook}), one has  \[  H^{2}(X,\mathcal O_X)=0.  \]

We also recall the following topological criterion theorem of domain of holomorphy from \cite[Theorem~14, Section~G]{Gun90}.

\textit{If $D$ is an open subset of $\C^n$ for which $H^{0,q}(D)=0$ for $1 \leq q \leq n-1$, then $D$ is a domain of holomorphy.}

In the present situation, this theorem implies that $H^1(X,\mathcal O_X)\neq 0$.

We now determine the Dolbeault cohomology with compact support.
Since \(X\) is non-compact, we have \[ H^0_c(X,\mathcal O_X) =H^0_c(X,\Omega^1_X) =H^0_c(X,\Omega^2_X) =0. \] Following \cite{Ser55}, consider the long exact sequence of cohomology with compact support associated with the decomposition \(\C^2 = X \cup (\C^2\setminus X)\): \[ H^0_c(\C^2\setminus X,\mathcal O) \longrightarrow H^1_c(X,\mathcal O_X) \longrightarrow H^1_c(\C^2,\mathcal O). \] Since \(\C^2\setminus X\) is non-compact connected, we have \(H^0_c(\C^2\setminus X,\mathcal O)=0\), and since \(\C^2\) is Stein, \[ H^1_c(\C^2,\mathcal O)=0, \]
by \cite[Théorème~3]{Ser55}. 
It follows that \[ H^1_c(X,\mathcal O_X)=0. \] 
The same conclusion holds for holomorphic differential forms, and hence \[ H^1_c(X,\mathcal O_X) =H^1_c(X,\Omega^1_X) =H^1_c(X,\Omega^2_X) =0. \]
Now we choose two different closed subsets to exhibit two distinct types of failure of duality.
\subsection{The complement of a real line}
Consider \(X\subset \C^2\) to be the complement of a real line. Then \(X\) is homotopy equivalent to \(S^2\). In particular, its Betti numbers are \[ b_0=1,\qquad b_1=0,\qquad b_2=1,\qquad b_3=0,\qquad b_4=0. \] 
By Poincaré duality, the corresponding Betti numbers with compact support are determined accordingly.

We now verify that \(X\) is not a domain of holomorphy. In fact, $ H^1(X,\mathcal O_X) $ is infinite dimensional as the Serre duality fails for $X$ (cf. \cite[Th\'eor\`eme 2, Proposition 6, Section 14]{Ser55}).

By the solution of the Levi problem (see, for example, \cite[Theorems~(6.11), (7.2) in Chapter~I and Theorem~(9.11) in Chapter~VIII]{agbook}), pseudoconvexity is equivalent to being a domain of holomorphy. Thus, to show that \(X\) is not a domain of holomorphy, it suffices to examine the Levi form of a defining exhaustion function. 
In this case, we consider the distance function to the boundary,
 \[  \phi(z_1,z_2)=-\log\bigl(x_1^2+|z_2|^2\bigr),  \qquad z_1=x_1+\sqrt{-1} y_1,\ z_2\in\C,  \] 
 and set  \[  \rho:=x_1^2+|z_2|^2  \]
such that the real line is defined by $\{x_1=z_2=0\}$. 
Recall that \[ \frac{\partial}{\partial z_1} =\tfrac12\Bigl(\frac{\partial}{\partial x_1} -\sqrt{-1}\frac{\partial}{\partial y_1}\Bigr), \qquad \frac{\partial}{\partial\bar z_1} =\tfrac12\Bigl(\frac{\partial}{\partial x_1} +\sqrt{-1}\frac{\partial}{\partial y_1}\Bigr). \] 
Since \(\rho\) is independent of \(y_1\), \[ \partial_{z_1}\rho=\partial_{\bar z_1}\rho=x_1, \qquad \partial_{z_2}\rho=\bar z_2, \quad \partial_{\bar z_2}\rho=z_2. \] 
The first derivatives of \(\phi\) are \[ \partial_{z_1}\phi=-\frac{x_1}{\rho}, \qquad \partial_{z_2}\phi=-\frac{\bar z_2}{\rho}. \] 
The second derivatives are computed as follows: \begin{align*} \phi_{1\bar1} &=\partial_{\bar z_1}\!\left(-\frac{x_1}{\rho}\right) =-\frac12\Bigl(\frac{1}{\rho}-\frac{2x_1^2}{\rho^2}\Bigr) =\frac{x_1^2-|z_2|^2}{2\rho^2},\\[4pt] \phi_{1\bar2} &=\partial_{\bar z_2}\!\left(-\frac{x_1}{\rho}\right) =\frac{x_1 z_2}{\rho^2},\\[4pt] \phi_{2\bar1} &=\partial_{\bar z_1}\!\left(-\frac{\bar z_2}{\rho}\right) =\frac{x_1\bar z_2}{\rho^2},\\[4pt] \phi_{2\bar2} &=\partial_{\bar z_2}\!\left(-\frac{\bar z_2}{\rho}\right) =-\frac{1}{\rho}+\frac{|z_2|^2}{\rho^2} =-\frac{x_1^2}{\rho^2}. \end{align*} 
Therefore, the complex Hessian matrix \((\phi_{j\bar k})\) is \[ (\phi_{j\bar k}) =\frac{1}{\rho^2} \begin{pmatrix} \dfrac{x_1^2-|z_2|^2}{2} & x_1 z_2\\ x_1\bar z_2 & -x_1^2 \end{pmatrix} \] 
with $\det(\phi_{j\bar k}) = -\frac{x_1^2}{2\rho^3}$.
In particular, the Levi form is not semi-positive everywhere; hence \(X\) is not pseudoconvex and hence not a domain of holomorphy.

We now return to the computation of Bott–Chern and Aeppli cohomology. In particular, we focus on the groups 
 \[  H^{1,1}_{A,c}(X)  \quad\text{and}\quad  H^{1,1}_{BC}(X).  \]

By the long exact sequence relating Bott–Chern cohomology to de Rham and Dolbeault cohomology (\ref{eq:les}), 
\[ H^{1}(X,\C) \longrightarrow H^{1}(X,\mathcal O_X\oplus \overline{\mathcal O}_X) \longrightarrow H^{1,1}_{BC}(X) \longrightarrow H^{2}(X,\C) \longrightarrow H^{2}(X,\mathcal O_X\oplus \overline{\mathcal O}_X), \] 
and using 
\[ H^{1}(X,\C)=0, \qquad H^{2}(X,\mathcal O_X)=0, \]
 it follows that 
 \[ \dim H^{1,1}_{BC}(X)\ge \dim H^{2}(X,\C)=1, \] 
 and in fact \[ \dim H^{1,1}_{BC}(X)\ge 3, \] since 
 \(H^{1}(X,\mathcal O_X)\neq 0\) (which is in fact infinite dimensional). 
 
 Next, consider Aeppli cohomology with compact support. From the long exact sequence 
 \[ \mathbb H^{2}_c(X,\mathcal O_X\to\Omega^1_X) \oplus \mathbb H^{2}_c(X,\overline{\mathcal O}_X\to\overline{\Omega}^1_X) \longrightarrow H^{1,1}_{A,c}(X) \longrightarrow H^{3}_c(X,\C), \] 
 and the fact that \(H^{3}_c(X,\C)=0\), we obtain a surjection onto \(H^{1,1}_{A,c}(X)\). 
 
 To estimate the hypercohomology term, consider the exact sequence 
 \[ H^{1}_c(X,\Omega^1_X) \longrightarrow \mathbb H^{2}_c(X,\mathcal O_X\to\Omega^1_X) \longrightarrow H^{2}_c(X,\mathcal O_X) \longrightarrow H^{2}_c(X,\Omega^1_X). \] 
 Since \(H^{1}_c(X,\Omega^1_X)=0\), this yields 
 \[ \mathbb H^{2}_c(X,\mathcal O_X\to\Omega^1_X) = \ker\!\bigl( H^{2}_c(X,\mathcal O_X) \to H^{2}_c(X,\Omega^1_X) \bigr). \] 
 
 Now consider the Hodge–to–de Rham spectral sequence with compact support. 
By bidegree reasons, all differentials \(d_r\) vanish for \(r\ge 3\) since \( \dim X = 2 \). In particular, the only potentially nonzero differential affecting \(E^{0,2}\) is \(d_2\). Moreover, 
 \[ \mathbb H^{2}_c(X,\mathcal O_X\to\Omega^1_X) = E^{0,2}_2. \] 
 Since \[ E^{2,1}_1 = H^1_c(X,\Omega^2_X)=0, \] we have \(E^{2,1}_2=0\), and hence the differential \[ d_2\colon E^{0,2}_2\to E^{2,1}_2 \] vanishes.
  Therefore, \[ E^{0,2}_2=E^{0,2}_\infty \] injects into \(H^{2}_c(X,\C)\), which is one-dimensional. 
  It follows that \[ \dim H^{1,1}_{A,c}(X)\le 2. \] Comparing dimensions, we conclude that \(H^{1,1}_{A,c}(X)\) and \(H^{1,1}_{BC}(X)\) cannot be topological duals of each other.

\subsection{The complement of a real plane}
We now present a variant of Serre’s example that complements the preceding discussion. While in the previous paragraph we considered Bott--Chern cohomology without compact support and Aeppli cohomology with compact support, the present example exhibits the opposite phenomenon. 

Let $X=\C^2\setminus\R^2$, where $\R^2\subset\C^2$ is the standard totally real plane. Then $X$ is homotopy equivalent to $S^1$, and hence its Betti numbers are \[ b_0=1,\qquad b_1=1,\qquad b_2=b_3=b_4=0. \] This domain can be written as the tube domain \[ X=(\R^2\setminus\{0\})+\sqrt{-1}\R^2. \] 
Since $\R^2\setminus\{0\}$ is not convex, Bochner’s tube theorem (see, for example, \cite[Theorem 2.5.10]{Hor90}) implies that $X$ is not a domain of holomorphy. 

In contrast to the previous example, we find that the Bott--Chern cohomology with compact support vanishes in bidegree $(1,1)$, whereas the corresponding Aeppli cohomology without compact support is nontrivial.

 Indeed, from the long exact sequence \[ \cdots \to H^{1}_c(X,\mathcal O_X\oplus\overline{\mathcal O}_X) \to H^{1,1}_{BC,c}(X) \to H^2_c(X,\C)\to\cdots, \] together with the vanishing \[ H^{1}_c(X,\mathcal O_X)=0, \qquad H^2_c(X,\C)=0, \] we deduce \[ H^{1,1}_{BC,c}(X)=0. \] 
 
 On the other hand, as in the preceding paragraph, Aeppli cohomology can be described in terms of hypercohomology of the complex $\mathcal O_X\to\Omega^1_X$. 
 
By the long exact sequence associated with Aeppli cohomology, \[ H^2(X,\C) \longrightarrow \mathbb H^{2}(X,\mathcal O_X \to \Omega^1_X) \oplus \mathbb H^2(X, \overline{\mathcal O}_X \to \overline{\Omega}^1_X) \longrightarrow H^{1,1}_{A}(X) \longrightarrow H^3(X,\C), \] and using the fact that \[ H^2(X,\C)=0, \qquad H^3(X,\C)=0, \] we obtain an isomorphism \[ H^{1,1}_{A}(X) \simeq \mathbb H^{2}(X,\mathcal O_X \to \Omega^1_X) \oplus \mathbb H^2(X, \overline{\mathcal O}_X \to \overline{\Omega}^1_X). \]

Next, consider the long exact sequence in hypercohomology associated with the complex \(\mathcal O_X \to \Omega^1_X\): \[ H^{1}(X,\mathcal O_X) \longrightarrow  H^{1}(X,\Omega^1_X) \longrightarrow \mathbb H^{2}(X,\mathcal O_X \to \Omega^1_X) \longrightarrow  H^{2}(X,\mathcal O_X). \] 
Since $H^{2}(X,\mathcal O_X)=0$, it follows that \[ \mathbb H^{2}(X,\mathcal O_X \to \Omega^1_X) = \mathrm{Coker}\bigl( H^{1}(X,\mathcal O_X) \to H^{1}(X,\Omega^1_X) \bigr). \] 

We now consider the Hodge--to--de~Rham spectral sequence associated with the Dolbeault complex of \(X\). By bidegree reasons, all differentials \(d_r\) vanish for \(r \ge 3\), so the spectral sequence degenerates at the \(E_3\)-page. Moreover, $ E_2^{1,1} = E_\infty^{1,1}, $ and this group injects as a subspace of \(H^2(X,\C)\). 
Since \(H^2(X,\C)=0\), it follows that \[ E_2^{1,1}=E_\infty^{1,1}=0. \] 
Thus we have the identification 
\[ \mathrm{Im}\bigl( H^{1}(X,\mathcal O_X) \to H^{1}(X,\Omega^1_X) \bigr) = \mathrm{Ker}\bigl( H^{1}(X,\Omega^1_X) \to H^{1}(X,\Omega^2_X) \bigr). \]
 Consequently, 
 \begin{align*} \mathrm{Coker}\bigl( H^{1}(X,\mathcal O_X) \to H^{1}(X,\Omega^1_X) \bigr) &= H^{1}(X,\Omega^1_X)\Big/\mathrm{Ker}\bigl( H^{1}(X,\Omega^1_X) \to H^{1}(X,\Omega^2_X) \bigr)\\ &\simeq \mathrm{Im}\bigl( H^{1}(X,\Omega^1_X) \to H^{1}(X,\Omega^2_X) \bigr). \end{align*} 
On the other hand, by definition of the \(E_1\)-page, \[ E_1^{0,2}=H^2(X,\mathcal O_X)=0, \] and hence $ E_2^{0,2}=0. $ 
Therefore the differential \[ d_2 \colon E_2^{0,2} \longrightarrow E_2^{2,1} \] vanishes identically, and we obtain \[ E_2^{2,1}=E_3^{2,1}=E_\infty^{2,1}=0. \]
 Finally, since \[ E_1^{2,1}=H^{1}(X,\Omega^2_X)\neq 0, \] and since no further differentials can kill this term, we have \[ \mathrm{Im}\bigl( H^{1}(X,\Omega^1_X) \to H^{1}(X,\Omega^2_X) \bigr) = \mathrm{Ker}\bigl( H^{1}(X,\Omega^2_X) \to H^{1}(X,\Omega^3_X) \bigr) = H^{1}(X,\Omega^2_X)\neq 0. \] Therefore, \[ H^{1,1}_{A}(X)\neq 0. \]
  
Department of Mathematics, Graduate School of Science, Osaka Metropolitan University, 3-3-138 Sugimoto, Osaka 558-8585, Japan.\\
Email address: xiaojun.wu@univ-cotedazur.fr; y25161q@omu.ac.jp.  
  \end{document}